\documentclass[11pt]{article}
\usepackage[T1]{fontenc}
\usepackage{amsfonts}
\usepackage{amsmath}
\usepackage{amssymb}
\usepackage{amsthm}
\usepackage{bbm}
\usepackage{bm}
\usepackage{mathrsfs}
\usepackage{verbatim}
\usepackage{setspace}
\usepackage{color}
\usepackage{pdfsync}
\usepackage{enumitem}
\usepackage{graphicx}
\usepackage{subfigure}
\usepackage{tikz}
\usetikzlibrary{patterns}
\usepackage[pdfborder={0 0 0}]{hyperref}
\hypersetup{
  urlcolor = black,
  pdfauthor = {Stephan Eckstein, Marcel Nutz},
  pdfkeywords = {Entropic Optimal Transport; f-Divergence; Regularization; Quantization},
  pdftitle = {Convergence Rates for Regularized Optimal Transport via Quantization},
  pdfsubject = {Convergence Rates for Regularized Optimal Transport via Quantization},
  pdfpagemode = UseNone
}
\usepackage[capitalize]{cleveref} %

\theoremstyle{plain}
\newtheorem{theorem}{Theorem}[section]
\newtheorem{proposition}[theorem]{Proposition}
\newtheorem{lemma}[theorem]{Lemma}
\newtheorem{corollary}[theorem]{Corollary}

\theoremstyle{definition}
\newtheorem{definition}[theorem]{Definition}
\newtheorem{remark}[theorem]{Remark}

\newtheorem{example}[theorem]{Example}

\theoremstyle{remark}

\renewenvironment{thebibliography}[1]{%
\begin{oldthebibliography}{#1}%
\setlength{\baselineskip}{1em}
\linespread{.2}
\small
\setlength{\parskip}{0.25ex}%
\setlength{\itemsep}{.20em}%
}%
{%
\end{oldthebibliography}%
}
\newcommand{\eps}{\varepsilon}

\newcommand{\N}{\mathbb{N}}

\newcommand{\R}{\mathbb{R}}

\newcommand{\cC}{\mathcal{C}}

\newcommand{\cP}{\mathcal{P}}

\DeclareMathOperator{\proj}{proj}

\DeclareMathOperator{\spt}{spt}

\DeclareMathOperator{\inv}{inv}

\DeclareMathOperator*{\argmin}{arg\, min}

\newcommand{\1}{\mathbf{1}}

\newcommand{\br}[1]{\langle #1 \rangle}

\newcommand{\mykill}[1]{}

\newcommand{\eins}{\1}%
\newcommand{\OT}{\mathsf{OT}}

\numberwithin{equation}{section}

\begin{document}

\title{\vspace{-2em}
  Convergence Rates for Regularized Optimal Transport via Quantization\thanks{The authors thank Guillaume Carlier, 
  L\'{e}na\"{\i}c Chizat, Nicolas Juillet, Harald Luschgy, Jon Niles-Weed, Gilles Pag\`es
   and Luca Tamanini
  for discussions and encouragement.}%
 }
\date{\today}

\author{
  Stephan Eckstein%
  \thanks{Department of Mathematics, ETH Zurich, seckstein@ethz.ch. 
  }
  \and
  Marcel Nutz%
  \thanks{
  Departments of Statistics and Mathematics, Columbia University, mnutz@columbia.edu. %
  Research supported by an Alfred P.\ Sloan Fellowship and NSF Grants DMS-1812661, DMS-2106056.}
  }
\maketitle \vspace{-1.2em}

\begin{abstract}
We study the convergence of divergence-regularized optimal transport as the regularization parameter vanishes. Sharp rates for general divergences including relative entropy or $L^{p}$ regularization, general transport costs and multi-marginal problems are obtained. A novel methodology using quantization and martingale couplings is suitable for non-compact marginals and achieves, in particular, the sharp leading-order term of entropically regularized 2-Wasserstein distance for marginals with finite $(2+\delta)$-moment.
\end{abstract}

\vspace{.3em}

{\small

}
\vspace{.6em}

\section{Introduction}

We study regularized optimal transport problems of the form
\begin{align*}
\OT_{f, \varepsilon} := \inf_{\pi \in \Pi(\mu_1, \dots, \mu_N)} \int c \,d\pi + \varepsilon D_f(\pi, \mu_{1}\otimes \cdots\otimes\mu_{N})
\end{align*}
where $D_{f}$ is an $f$-divergence, for example relative entropy (Kullback--Leibler divergence) or $L^{p}$ regularization. %
(Notation is detailed in \cref{se:prelims}.) Note that $\eps=0$ yields the classical optimal transport problem~$\OT$ without regularization. We are interested in the speed of convergence $\OT_{f, \varepsilon}\to\OT$ as the regularization parameter~$\eps$ tends to zero, especially its dependence on the marginals~$\mu_{i}$ and the divergence~$D_{f}$. 

Regularized optimal transport has attracted a great deal of research in recent years, chiefly because regularization enables the use of efficient numerical algorithms (e.g, \cite{BlanchetJambulapatiKentSidford.18,Cuturi.13,LinHoJordan.19,CuturiPeyre.19} and the references therein) to approximate~$\OT$ in high-dimensional applications---whence the interest in the speed of convergence. The most important divergence is relative entropy which gives rise to Sinkhorn's algorithm (or IPFP); here $\OT_{f, \varepsilon}$ is often called the \emph{entropic optimal transport} problem (e.g., \cite{Nutz.20,CuturiPeyre.19}). Other divergences, especially $L^{p}$ regularization, are being used in applications where sparse optimizers are desired or weak penalization (small $\eps$) causes numerical instabilities with entropic regularization~\cite{blondel18quadratic,DiMarinoGerolin.20b,EssidSolomon.18,LorenzMannsMeyer.21,TerjekGonzalez.22}.
For multi-marginal transport and the related Wasserstein barycenters, see for instance~\cite{AguehCarlier.11, BenamouCarlierNenna.19, Carlier.21, CarlierEichingerKroshnin.20, CarlierLaborde.20}. Literature more specific to the convergence $\OT_{f, \varepsilon}\to\OT$ is discussed below.

In this paper, we propose a novel methodology to estimate $\OT_{f, \varepsilon} - \OT$ based on quantization. It is simultaneously more general and, arguably, easier than previous arguments, allowing us to obtain convergence rates for a wide class of $f$-divergences, unbounded cost functions and multi-marginal problems in a unified manner---the methodology may be as important as the results themselves. Even for entropic optimal transport with two marginals and quadratic cost, we substantially improve on the existing results, by allowing for arbitrary marginals with finite $(2+\delta)$-moments where previous techniques required compact supports and uniformly bounded densities \cite{CarlierPegonTamanini.22,Chizat2020Faster,ConfortiTamanini.19,Pal.19}.

To give an informal preview, let us focus on $N=2$ marginals with entropic or $L^{p}$ regularization ($p>1$) for simplicity. In those examples, we obtain non-asymptotic bounds of the form\begin{align*}
	\OT_{f, \varepsilon} - \OT &\leq \beta \varepsilon \log\left(\frac{1}{\varepsilon}\right) + K \varepsilon  &&\mbox{for entropic regularization}, \\
	\OT_{f, \varepsilon} - \OT &\leq K\eps^{\frac{1}{(p-1)\beta+1}}  &&\mbox{for $L^{p}$ regularization},
\end{align*}
where $\beta$ reflects a certain quantization dimension. In our first result (\cref{thm:lip}), $\beta$ encodes the optimal quantization rate for one of the marginals---if $\mu_{i}$ are measures on $\R^{d_{i}}$, this leads to $\beta\leq d_{1}\wedge d_{2}$. In this result, we assume that the integrated cost $\pi\mapsto \int c\, d\pi$ is Lipschitz when restricted to a certain set of couplings; this is satisfied for Lipschitz functions~$c$ but also, e.g., for $|x-y|^{p}$ with $p\geq1$ on $\R^{d}\times\R^{d}$. The stated estimates are sharp in certain examples (see Section~\ref{se:sharpness}), up to the constant~$K$.

The key idea is to use so-called \emph{shadows} to transfer explicit divergence bounds for discrete measures into continuous couplings with controlled divergence, while also bounding the Wasserstein distance. As quantization theory has long studied how fast general measures can be approximated with discrete ones, this enables us to control both the transport and divergence terms in~$\OT_{f, \varepsilon}$. Specifically, a rate is found by choosing the number of points for the quantization of the marginals relative to the regularization parameter~$\eps$ such as to balance the transport and divergence terms. At a high level, the shadow construction is a substitute for the widely used \emph{block approximation} method first introduced in~\cite{CarlierDuvalPeyreSchmitzer.17}. Employing quantization and Wasserstein geodesics instead of building blocks explicitly, our construction fully exploits the flexibility of the $p$-Wasserstein distance, making it very suitable for unbounded domains and costs. 

Our main result (\cref{thm:second}) pertains to cost functions on $\R^{d}\times \cdots\times\R^{d}$ admitting a bounded second derivative, in particular the quadratic cost, and improves the value of $\beta$ to $d/2$ under sufficient regularity. Here, smoothness leads to the factor $1/2$ while $d$ reflects the quantization rate for an optimal transport plan (of the unregularized problem~$\OT$) rather than the marginals.
The key idea is a martingale argument that seems to be novel: the martingale property of 2-Wasserstein quantization can be used to eliminate the first-order term in the integrated Taylor expansion of the cost function. The remaining leading term is then of second order, whence the factor~$1/2$. Once again, the martingale methodology lends itself to the unbounded setting; moreover, the rates are sharp in wide class of examples.
In particular, we establish the leading-order term $\frac{d}{2} \varepsilon \log\left(\frac{1}{\varepsilon}\right)$ for entropically regularized 2-Wasserstein distance whenever the marginals have finite moments of order $2+\delta$ for some $\delta>0$ (\cref{co:quadratic}). In its proof, Minty's trick~\cite{Minty.62} is used to establish the quantization rate for an optimal transport plan.

For discrete problems, the study of entropic regularization and its convergence goes back to~\cite{CominettiSanMartin.94}; see also~\cite{Weed.18} for a non-asymptotic result, \cite{AltschulerNilesWeedStromme.21} for a semi-discrete problem, and \cite{AltschulerBoixAdsera.22} for multi-marginal transport. Here, we are mainly interested in continuous problems. As $\OT_{f, \varepsilon} - \OT = O(\eps)$ if and only if there exists an optimal transport with finite divergence (\cref{pr:finiteDivergenceOT}) and as the latter typically fails for continuous marginals, we shall be dealing with convergence slower than~$O(\eps)$. In the continuous case, we are not aware of works addressing the multi-marginal problem, and for two marginals, almost all results are on the entropic regularization; an exception is~\cite{Bianco.21} where $\chi^{2}$ divergence is studied in a compact setting and an upper bound of order $\eps^{1/(d+1)}$ is found. Returning to the entropic case, 
the link between $\OT_{f, \varepsilon}$ and $\OT$ goes back to~\cite{Mikami.02, Mikami.04} in the Schr\"odinger bridge problem (which is closely related to entropic optimal transport with quadratic cost; cf.~\cite{Leonard.14}). Gamma-convergence was shown in~\cite{Leonard.12}; see also~\cite{CarlierDuvalPeyreSchmitzer.17} for a proof in a setting closer to ours. A stochastic control viewpoint is presented in~\cite{ChenGeorgiouPavon.16}. Early quantitative results for quadratic cost, from a large deviations viewpoint, are \cite{AdamsDirrPeletierZimmer.11,DuongLaschosRenger.13,ErbarMaasRenger.15}, 
later extended in~\cite{Pal.19} to cost functions closely modeled on the quadratic. While these are first-order results, a second-order expansion of the optimal cost was obtained in \cite{ConfortiTamanini.19} for the Schr\"odinger bridge setting and in \cite{Chizat2020Faster} for entropic optimal transport, all with quadratic cost. These results require strong regularity assumptions in addition to compactly supported marginals.

The most comparable results by far were obtained in the very recent (and partly concurrent) work~\cite{CarlierPegonTamanini.22} which addresses general cost functions and obtains rates similar to ours, at least for compactly supported marginals, in the case of entropic regularization with two marginals. Remarkably, the methods used are quite different. For Lipschitz cost functions and compactly supported marginals, \cite[Proposition~3.1]{CarlierPegonTamanini.22} finds that $\OT_{f, \varepsilon} - \OT \leq d \varepsilon \log(1/\eps) + O(\varepsilon)$ where $d$ is the minimum of the two marginal dimensions. A potentially more general result is obtained with a notion of upper R\'enyi dimension of the marginals, however a more concrete bound is only available through the box dimension which requires compactness to be finite.\footnote{Note added in proof: This statement referred to the preprint version of~\cite{CarlierPegonTamanini.22}. The final published version provides an improved result; namely, the upper R\'enyi dimension is bounded by the Euclidean dimension as soon as the marginal has a finite logarithmic moment.} The proof proceeds through a block approximation, applying the Lipschitz property on each block. Our \cref{thm:lip} (specialized to the entropic divergence on two marginals) obtains a bound of the same form but with the dimension defined by quantization. Using $p$-Wasserstein distance with finite~$p$, quantization is well-behaved also for unbounded domains, so that the bound can be established for general marginals with finite $(p+\delta)$-moments. Moreover, \Cref{thm:lip} applies to costs like~$|x-y|^{p}$, $p\geq1$ as the Lipschitz property is only required in an integrated form. Shadows are a convenient and robust tool in this context, as is also exemplified by their application to adapted (causal) optimal transport in~\cite{EcksteinPammer.22}.

For  cost functions of class $C^{1,1}$ (thus with a.e.\ bounded second derivative) and compactly supported marginals with uniformly bounded Lebesgue densities, \cite[Proposition~3.4]{CarlierPegonTamanini.22}  shows that $\OT_{f, \varepsilon} - \OT \leq \frac{d}{2} \varepsilon \log\left(\frac{1}{\varepsilon}\right) + O(\varepsilon)$. The proof is deep and based on the fine regularity of the Kantorovich potential; namely, a quadratic bound on the integrated difference between a $\lambda$-convex function and its first-order Taylor expansion \cite[Lemma~3.6]{CarlierPegonTamanini.22}. This bound depends directly on the diameter of the domain and the density assumption is needed to pass from Lebesgue measure to the actual marginals. By contrast, the martingale argument used for our \cref{thm:second} applies to unbounded domains and is fairly robust; for instance, it easily extends to the multi-marginal case. It does, however, take as its input the quantization rate of an optimal transport plan~$\pi^{*}$, so that it needs to be applied together with a regularity result for~$\pi^{*}$. For quadratic cost, we prove that the rate is indeed $1/d$ in great generality, assuming only finite moments of order $2+\delta$. For compactly supported marginals, a quite generic sufficient condition for this rate is the nondegeneracy of the cost; i.e., invertibility of the mixed derivative $D^{2}_{xy}c(x,y)$. For unbounded but sufficiently integrable marginals, we show a rate arbitrarily close to $1/d$ if nondegeneracy holds in a uniform sense. 

In~\cite{CarlierPegonTamanini.22}, the authors also obtain a matching lower bound for the convergence rate (for entropic regularization), for cost functions satisfying the aforementioned nondegeneracy condition and sufficiently regular marginals. The proof is again based on a fine analysis of the Kantorovich potential. The key tool is a quadratic detachment estimate \cite[Lemma~4.2]{CarlierPegonTamanini.22} which we reuse in \cref{se:sharpness} to obtain matching lower bounds for $L^{p}$ regularization as well.

While the present work focuses on the convergence of the optimal cost $\OT_{f, \varepsilon}$, two related question are the convergence of the optimal couplings and optimal dual potentials. See ~\cite{BerntonGhosalNutz.21,CarlierDuvalPeyreSchmitzer.17,Leonard.12, Leonard.14}  and~\cite{Berman.20,ChiariniConfortiGrecoTamanini.22,GigliTamanini.21,NutzWiesel.21,PooladianWeed.21}, respectively, and the references therein. As seen in~\cite{BerntonGhosalNutz.21,ChiariniConfortiGrecoTamanini.22}, the convergence is also related  to the stability of $\OT_{f, \varepsilon}$ wrt.\ the marginals \cite{CarlierLaborde.20, EcksteinNutz.21, GhosalNutzBernton.21b, NutzWiesel.22}.

The remainder of this paper is organized as follows. \Cref{se:prelims} formally introduces the problem and notation, then gathers preliminaries on quantization, divergence bounds for discrete couplings, and shadows. \Cref{se:main} contains the main results on convergence rates. \Cref{se:sharpness} provides instances where the rates are sharp and Appendix~\ref{se:appendix} gathers two additional results.

\section{Preliminaries}\label{se:prelims}

\subsection{Setting and Notation}

Let $(Y,d_{Y})$ be a Polish space and $\mathcal{P}(Y)$ its set of Borel probability measures. Fix $p\in[1,\infty)$ and denote by $\mathcal{P}_p(Y)$ the subset of measures $\mu$ with finite $p$-th moment; i.e., $\int d_{Y}(x,\hat{x})^p \,\mu(dx) < \infty$ for some (and then all) $\hat{x} \in Y$. The $p$-Wasserstein distance $W_p(\mu, \nu)$ between $\mu,\nu \in \mathcal{P}_p(Y)$ is defined via
\begin{align*}
W_p(\mu, \nu)^p &= \inf_{\pi \in \Pi(\mu, \nu)} \int d_{Y}(x, y)^p \,\pi(dx, dy).
\end{align*}
Fix $N\in \N$ and let $(X_i, d_{X_i})$, $i=1, \dots, N$ be Polish probability spaces with measures $\mu_{i}\in\cP(X_{i})$. We denote by $X = \prod_{i=1}^N X_i$ the product space and use the particular product metric 
$%
d_{X, p}(x, y):=
(\sum_{i=1}^N d_{X_i}(x_i, y_i)^p)^{1/p} %
$ %
to induce the $p$-Wasserstein distance on~$X$.

Let $c: X\to\R$ be continuous with growth of order $p$; i.e.,
\begin{equation*}
|c(x)|\leq C \big(1+d_{X, p}(x,\hat{x})^p\big)
\end{equation*}
for some $C\geq0$ and $\hat{x} \in X$. The optimal transport problem is 
\begin{align*}
\OT &:= \inf_{\pi \in \Pi(\mu_1, \dots, \mu_N)} \int c \,d\pi,
\end{align*}
where $\Pi(\mu_1, \dots, \mu_N)\subset \cP_{p}(X)$ denotes the set of couplings of the marginal measures~$\mu_{i}\in\cP_{p}(X_{i})$. The growth of~$c$ ensures that $\OT$ is finite.

Let $f : \mathbb{R}_+ \rightarrow \mathbb{R}$ be a strictly convex, lower bounded function  with $f(1) = 0$ and $\lim_{x \rightarrow \infty} f(x)/x = \infty$. The \emph{$f$-divergence} $D_f(\mu,\nu)$ between probabilities $\mu,\nu$ on a common space is defined by
\[
D_f(\mu, \nu) := \int f\left(\frac{d\mu}{d\nu}\right) \,d\nu \quad\mbox{for}\quad \mu\ll \nu
\]
and $D_f(\mu, \nu) :=\infty$ for $\mu\not\ll\nu$. The $D_{f}$-regularized  transport problem is
\begin{align*}
\OT_{f, \varepsilon} &:= \inf_{\pi \in \Pi(\mu_1, \dots, \mu_N)} \int c \,d\pi + \varepsilon D_f(\pi, P), \quad P:= \mu_{1}\otimes \cdots\otimes\mu_{N},
\end{align*}
where $\eps>0$ is the regularization parameter. In particular, entropic optimal transport corresponds to $f(x)=x\log(x)$.

\subsection{Quantization}

On a Polish space $Y$, we denote by $\mathcal{P}^n(Y) \subset \mathcal{P}(Y)$ the set of probability measures supported on at most $n$ points. Given $p \in [1, \infty)$ and $\mu \in \mathcal{P}_p(Y)$, our results depend on an approximation rate of the form
	\begin{equation}\label{eq:quant}
	\exists\, \mu^n \in \mathcal{P}^n(Y): \quad W_p(\mu^n, \mu) \leq C n^{-\alpha}, \quad n\geq1\tag{${\sf quant}_{p}(C,\alpha)$}
	\end{equation}
  for constants $C\geq0$ and $\alpha>0$. The takeaway of the following is that if the support of $\mu$ is $d$-dimensional, this property typically holds with $\alpha=1/d$.

\begin{remark}[Quantization rate on~$\R^{d}$]\label{rk:quantizationRate}
  Let $Y=\R^{d}$. If $\mu \in \mathcal{P}_{p+\delta}(Y)$ for some $\delta>0$, then~${\sf quant}_{p}(C,\alpha)$ holds with $\alpha=1/d$ for some $C\geq0$. More precisely, \cite[Theorem~6.2]{GrafLuschgy.00} shows that the exact asymptotic constant 
  $$
    C_{a}:=\lim_{n\to\infty} n^{1/d} \inf_{\mu^n \in \mathcal{P}^n(\R^{d})} W_p(\mu^n, \mu)
  $$
  can be expressed through a dimensional constant related to the $p$-quantization of the uniform measure on the unit cube and a moment of the density of the absolutely continuous part of $\mu$. In particular, $C_{a}>0$ as soon as $\mu$ is not mutually singular wrt.\ Lebesgue measure, showing that the rate $\alpha=1/d$ is then optimal. A bound for the (non-asymptotic) constant~$C$ in~${\sf quant}_{p}(C,\alpha)$ is given in \cite[Corollary~6.7]{GrafLuschgy.00}; its proof yields an explicit constant valid for all $n\geq 1$
  depending only on $p,\delta,d$ and $\int |x|^{p+\delta} \,\mu(dx)$.\footnote{The result in \cite[Corollary 6.7]{GrafLuschgy.00} is stated for all $n\geq C_{3}$ instead of $n\geq1$, for a certain constant $C_{3}$, in order to have a statement whose constants do not depend on the moment $\int |x|^{p+\delta} \,\mu(dx)$.  For our purposes, we do not mind such a dependence, and we can easily deduce a result valid for all $n\geq1$ by adjusting the constants.
}

\end{remark}  

For some variations of our results (in fact, only in the multi-marginal case of \cref{thm:lip} with non-entropic divergence), we use a slightly stronger notion, sometimes called (deterministic) empirical quantization, where the approximating measures are required to be uniform. Let $\mathcal{P}^{n, em}(Y)\subset \mathcal{P}(Y)$ be the set of uniform measures on $n$ points; i.e., measures $\mu^{n}=n^{-1}\sum_{i=1}^{n} \delta_{y_{i}}$ for some $y_{i}\in Y$. Similarly as above, we introduce
\begin{equation}\label{eq:empquant}
	\exists\, \mu^n \in \mathcal{P}^{n, em}(Y): \quad W_p(\mu^n, \mu) \leq C n^{-\alpha}, \quad n\geq1\tag{${\sf quant}_{p}^{em}(C,\alpha)$}
	\end{equation}
  for constants $C\geq0$ and $\alpha>0$. This condition clearly implies~${\sf quant}_{p}(C,\alpha)$; but at least in the high-dimensional regime, the optimal rate is in fact the same, as summarized in the following remark.
  
\begin{remark}[Empirical quantization rate on~$\R^{d}$]\label{rk:empiricalQuantizationRate}
  Let $Y=\R^{d}$. The well known \cite[Theorem~1]{FournierGuillin.15} shows, among other things, that if $\mu \in \mathcal{P}_{2p+\delta}(\mathbb{R}^d)$ with $d > 2p$, then~${\sf quant}_{p}^{em}(C,\alpha)$ holds with $\alpha = 1/d$ and a constant $C$ depending only on~$d,p,\delta$ and the $(2p+\delta)$-moment of $\mu$. In particular, this bound for the empirical rate coincides with the bound $1/d$ given for (arbitrary) quantization in Remark \ref{rk:quantizationRate}. Rates for other regimes ($d \leq 2p$) are also obtained in \cite[Theorem~1]{FournierGuillin.15}. Notably, the rates derived in \cite{FournierGuillin.15} are not based on a deterministic construction of $\mu^n$, but hold a.s.\ when $\mu^{n}$ are i.i.d.\ samples of~$\mu$. More precise constants for this result, and non-asymptotic bounds, can be found in the very recent work \cite{Fournier.22}. Rates for i.i.d.\ samples of measures supported on compact submanifolds are studied in \cite{WeedBach.19}.
  
  For measures with bounded support, a deterministic construction in \cite[Theorem~3]{Chevallier.18} provides the rate $\alpha=1/d$ and an explicit constant $C$ for $p<d$; for $p=d$, a logarithmic correction is added, whereas for $p>d$, the rate is at least $\alpha=1/p$. For unbounded measures, \cite[Corollary~1]{Chevallier.18} shows a slightly looser bound for the rate under the condition $\mu \in \mathcal{P}_{p+\delta}(Y)$. The univariate case $d=1$ has been studied in detail~\cite{BencheikhJourdain.22,XuBerger.19}. Here the optimal rate is $\alpha=1$ if $\mu$ has a positive density on its support and is sufficiently integrable, whereas $\alpha<1$ is known in several other cases (see~\cite[Table~1]{BencheikhJourdain.22} for an  overview). 
\end{remark}  

\subsection{Elementary Divergence Bounds}\label{se:divBounds}

For our purposes, discrete measures are useful because they admit straightforward divergence bounds. The best-known example is that a coupling $\pi\in\Pi(\mu_{1},\mu_{2})$ of marginals $\mu_{i}$ supported on~$n$ points has relative entropy $D_{f}(\pi,\mu_{1}\otimes\mu_{2})\leq \log n$. The following lemma collects some extensions of that fact for later reference. We recall that $\mathcal{P}^{n}(X_i)$ denotes the probabilities supported on at most~$n$ points, $\mathcal{P}^{n, em}(X_i)$ the  empirical measures on~$n$ points, and $P= \mu_{1}\otimes \cdots\otimes\mu_{N}$.

\begin{lemma}[Divergence bounds]
	\label{lem:boundfdiv}
	Let $\pi \in \Pi(\mu_1, \dots, \mu_N)$ and define~$\varphi$ by $f(x) = x \varphi(x)$.  Assume that $\varphi$ is nondecreasing.
	\begin{itemize}
		\item[(i)] If $N=2$, $\varphi$ is concave, and $\mu_2 \in \mathcal{P}^{n_2}(X_2)$, then
		$
		D_f(\pi, P) \leq \varphi(n_2).
		$
		\item[(ii)] If $\mu_i \in \mathcal{P}^{n_i, em}(X_i)$ for $i=2, \dots, N$, then
		$
		D_f(\pi, P) \leq \varphi \big(\prod_{i=2}^N n_i \big).
		$
		\item[(iii)] If $\varphi(x) = \log(x)$ and $\mu_i \in \mathcal{P}^{n_i}(X_i)$ for $i=2, \dots, N$, then
		$
		D_f(\pi, P) \leq \sum_{i=2}^N \log(n_i).
		$
	\end{itemize}

	\begin{proof}
	Denote by $\pi_{2:N}$ the marginal of~$\pi$ on $X_2 \times \dots \times X_N$. In particular, $P_{2:N} = \mu_2 \otimes \dots \otimes \mu_N$. We similarly define $\pi_{1:N-1}$ and $P_{1:N-1}$ as the marginals on $X_1 \times \dots \times X_{N-1}$. Let $\sigma$ be the counting measure on the (finite) support of~$P_{2:n}$.
    Disintegrating $\pi=\mu_{1}\otimes K$, we then have
    $
    \frac{d\pi}{dP}=\frac{dK}{dP_{2:n}}\leq \frac{d\sigma}{dP_{2:n}},
    $
     hence
    \begin{align*}%
      D_f(\pi, P) = \int \varphi\left(\frac{d\pi}{dP}\right) \,d\pi \leq \int \varphi\left(\frac{d\sigma}{dP_{2:N}}\right) \,d\pi. 
    \end{align*} 
		In the case~(i) where $N=2$, Jensen's inequality yields
		\[
		\int \varphi\left(\frac{d\sigma}{dP_{2:N}}\right) \,d\pi =  \int \varphi\left(\frac{d\sigma}{d\mu_2}\right) \,d\mu_2 \leq \varphi(n_2).
		\]
		Whereas in~(ii), $\frac{d\sigma}{dP_{2:N}}$ is constant and thus
		$
		  \int \varphi\big(\frac{d\sigma}{dP_{2:N}}\big) \,d\pi = \varphi\big(\prod_{i=2}^N n_i\big).
		$
		To see~(iii), we write $\frac{d\pi}{dP} = \frac{d\pi}{d(\pi_{1:N-1} \otimes \mu_N)} \frac{d(\pi_{1:N-1} \otimes \mu_N)}{dP}= \frac{d\pi}{d(\pi_{1:N-1} \otimes \mu_N)} \frac{d(\pi_{1:N-1})}{dP_{1:N-1}}$.
		As $\varphi(x) = \log(x)$, this yields
		\[
		D_f(\pi, P) = D_f(\pi, \pi_{1:N-1}\otimes \mu_N) + D_f(\pi_{1:N-1}, P_{1:N-1}).
		\]
		To bound the first term, we apply~(i) with $\mu_N$ as second marginal,
		\[
		D_f(\pi, P) \leq \log(n_N) + D_f(\pi_{1:N-1}, P_{1:N-1}).
		\]
		Iterating this argument yields
		$
		D_f(\pi, P) \leq \sum_{i=2}^N \log(n_i)
		$, which was the claim.
	\end{proof}
\end{lemma}

\subsection{Shadows}\label{se:shadows}

Given $\pi \in \Pi(\mu_1, \dots, \mu_N)$, the shadow $\tilde\pi$ of~$\pi$ on another vector $(\tilde\mu_1, \dots, \tilde\mu_N)$ of marginals is a particular $W_{p}$-projection of~$\pi$ onto $\Pi(\tilde\mu_1, \dots, \tilde\mu_N)$ that enjoys a control on its divergence. %
Intuitively, for $N=2$, the shadow $\tilde\pi$ is obtained by concatenating three transports: move $\tilde\mu_{1}$ to $\mu_{1}$ using a $W_{p}$-optimal transport, then follow the transport $\pi$ moving $\mu_{1}$ into $\mu_{2}$, and finally move $\mu_{2}$ to $\tilde\mu_{2}$ using a $W_{p}$-optimal transport. The general definition follows.

\begin{definition}[\cite{EcksteinNutz.21}]\label{de:shadow}
  Let $p\in [1,\infty)$ and $\mu_i, \tilde{\mu}_i \in \mathcal{P}_p(X_i)$,  $i=1, \dots, N$. Let $\kappa_{i} \in \Pi(\mu_i, \tilde\mu_i)$ be a coupling attaining $W_{p}(\mu_i,\tilde\mu_i)$ and  $\kappa_{i}=\mu_i \otimes K_i$ a disintegration.
  Given $\pi \in \Pi(\mu_1, \dots, \mu_N)$, its \emph{shadow} $\tilde\pi$ on $(\tilde\mu_1, \dots, \tilde\mu_N)$ is defined as the second marginal of $\pi \otimes K\in\cP(X\times X)$, where the kernel $K : X \rightarrow \mathcal{P}(X)$ is defined as
	$%
	K(x) = K_1(x_1) \otimes \dots \otimes  K_N(x_N). 
	$%
\end{definition}

The definition and the data processing inequality readily imply the following properties; see \cite[Lemma~3.2]{EcksteinNutz.21} for a detailed  proof.

\begin{lemma}[Shadow bounds]\label{le:shadow}
		Let $p\in [1,\infty)$ and $\mu_i, \tilde{\mu}_i \in \mathcal{P}_p(X_i)$,  $i=1, \dots, N$. Given $\pi \in \Pi(\mu_1, \dots, \mu_N)$, its shadow $\tilde{\pi} \in \Pi(\tilde{\mu}_1, \dots, \tilde{\mu}_N)$ satisfies
	\begin{align*}
	W_p(\pi, \tilde{\pi})^{p} &= %
	\sum_{i=1}^N W_p(\mu_i, \tilde{\mu}_i)^p,\\
	D_f(\tilde{\pi}, \tilde\mu_1 \otimes \dots \otimes \tilde\mu_N) &\leq D_f(\pi, \mu_1 \otimes \dots \otimes \mu_N).
	\end{align*}
\end{lemma}

\section{Main Results}\label{se:main}

One novel idea in this paper is to use a ``double shadow'' through auxiliary discrete marginals to approximate a given (typically singular) transport plan with one that has controlled divergence. 
To illustrate this, we start by re-proving the (known) convergence $\OT_{f, \varepsilon} \to \OT$ in our general setting.

\begin{proposition}\label{pr:qualitative}
  Let $p \in [1, \infty)$ and $\mu_i \in \mathcal{P}_p(X_i)$ for $i=1, \dots, N$. If $c$ is continuous with growth of order $p$, then $\lim_{\eps\to0}\OT_{f, \varepsilon} = \OT$.
\end{proposition}

\begin{proof}
	Using tightness of $\{\mu_{i}\}$, we can construct measures $\mu_{i}^{n}$ supported on~$n$ points with $W_p(\mu_{i}^{n}, \mu_{i})\to0$ for $i=1, \dots, N$. Let $\pi^*\in\Pi(\mu_1, \dots, \mu_N)$ be an optimizer of $\OT$. We introduce another coupling $\pi^{n}\in\Pi(\mu_1, \dots, \mu_N)$ as follows: first, let $\tilde{\pi}$ be the shadow of $\pi^*$ onto $(\mu_1^{n}, \mu_2^{n}, \dots, \mu_N^{n})$; then, define $\pi^{n}$ as the shadow of $\tilde{\pi}$ onto $(\mu_1, \dots, \mu_N)$. Using the triangle inequality and \cref{le:shadow}, this implies
  \begin{align*}
  	W_p(\pi^{n}, \pi^{*}) 
  	&\leq W_p(\pi^{n}, \tilde{\pi}) + W_p(\tilde{\pi}, \pi^{*})
  	 \leq 2 \left(\sum_{i=1}^N W_p(\mu_{i}^{n}, \mu_{i})^p\right)^{1/p}\to0.
	\end{align*}
	As $c$ is continuous with growth of order $p$, we conclude
	$
	\int c \,d\pi^{n} \to \int c \,d\pi^*.
	$
  On the other hand, \cref{le:shadow} yields
  \begin{align*}
    D_f(\pi^{n}, P) \leq D_f(\tilde{\pi}, \mu_1^{n} \otimes \mu_2^{n} \otimes \dots \mu_N^{n})<\infty,
	\end{align*}
	where the finiteness is trivial by discreteness of $\mu_i^{n}$. %
  Given $\delta>0$, choose~$n$ such that $\int c \,d\pi^{n} - \int c \,d\pi^* \leq\delta$, and then $\eps_{0}>0$ such that $\eps_{0} D_f(\pi^{n}, P)\leq \delta$. As $\pi^{n}$ is an admissible coupling for $\OT_{f, \varepsilon}$, we have shown $\OT_{f, \varepsilon} - \OT \leq 2\delta$ for all $\eps\leq \eps_{0}$. 
\end{proof}

\subsection{Rate for Lipschitz-type Costs}

To enable a quantitative version of \cref{pr:qualitative}, we need to control the speed of convergence $\int c \,d\pi^{n} \to \int c \,d\pi^*$ in its proof. We introduce the following adaptation of the condition~$({\rm A}_{L})$ of~\cite{EcksteinNutz.21}, stating that the integrated transport cost is Lipschitz with respect to the coupling.

\begin{definition}\label{de:ccond}
	Let $p\in [1,\infty)$ and $\mu_i \in \mathcal{P}_p(X_i)$,  $i=1, \dots, N$. Given constants $L, C\geq0$, we say that $c$ satisfies~\eqref{eq:ccond} if for all $\tilde\mu_i \in \mathcal{P}_p(X_i)$ with $W_p(\tilde{\mu}_i, \mu_i) \leq C$, $i=1, \dots, N$, we have
	\begin{equation}\label{eq:ccond}
	\left| \int c \, d(\pi- \tilde\pi)\right| \leq  LW_{p}(\pi,\tilde\pi) \tag{${\rm A}_{L, C}$}
	\end{equation}
	for all $\pi \in \Pi(\mu_1, \dots, \mu_N)$ and $\tilde\pi \in \Pi(\tilde\mu_1, \dots, \tilde\mu_N)$.
\end{definition}

Clearly~\eqref{eq:ccond} is satisfied (for all $C$) if $c$ is $L$-Lipschitz, but as discussed in \cite[Example~3.6]{EcksteinNutz.21}, the condition also captures various non-Lipschitz costs, like $c(x_{1},x_{2})=|x_{1}-x_{2}|^{p}$ on $\R^{d}\times\R^{d}$ with $p\in[1,\infty)$. In that case, the constant $L$  depends on the moments of the~$\mu_{i}$ and on~$C$. (The condition does not capture $|x_{1}-x_{2}|^{r}$ for $0<r<1$. An extension with a modulus of continuity instead of a Lipschitz constant is discussed in \cref{rk:modOfCont}.)

\begin{theorem}\label{thm:lip}
Let $p \in [1, \infty)$ and $\mu_i \in \mathcal{P}_p(X_i)$ for $i=1, \dots, N$.
Assume that $\mu_i$ satisfies ${\sf quant}_{p}(C,\alpha_i)$ for $i=2, \dots, N$ and that~$c$ satisfies~\eqref{eq:ccond}, for some $\alpha_2, \dots, \alpha_N \in(0,1]$ and $L,C \geq 0$.\footnote{Exponents $\alpha_{i}>1$ could be accommodated with minor changes in the constants below. In view of \cref{{rk:quantizationRate}}, the condition $\alpha_{i}\leq1$ is not restrictive in practice.} 
\begin{itemize}
	\item[(i)] Let $f(x) = x \log(x)$. Then for all $\varepsilon \in (0, 1]$,
	\[
	\OT_{f, \varepsilon} - \OT \leq \left(\sum_{i=2}^N \frac{1}{\alpha_i}\right) \varepsilon \log\left(\frac{1}{\varepsilon}\right) + 4 (N-1)^{1/p} L C \varepsilon.
	\]

	\item[(ii)] Let $f(x) = x \varphi(x)$, $\beta = \sum_{i=2}^N \frac{1}{\alpha_i}$, $\tilde{f}(x) = x \varphi(x^\beta)$. Assume that for some $x_{0},y_{0}\geq0$, $\tilde{f}$ is strictly increasing on $[x_0, \infty)$ with inverse  $\tilde{f}_{\inv}$ and $\varphi$ is nondecreasing. Suppose also that either $N=2$ and $\varphi$ is concave, or 
	the $\mu_i$ satisfy ${\sf quant}^{em}_{p}(C,\alpha_i)$ instead of ${\sf quant}_{p}(C,\alpha_i)$.
	Set $S_{\varepsilon} = \tilde{f}_{\inv}(\frac{1}{\varepsilon})$, which satisfies $\lim_{\varepsilon\to0} S_{\varepsilon}=\infty$ and $\lim_{\varepsilon\to0} \varepsilon S_{\varepsilon}=0$.   Then for all $\varepsilon \in [0, 1/x_0]$ small enough such that $S_\varepsilon \geq y_0^{1/\beta}+1$,
	\[
	\OT_{f, \varepsilon} - \OT \leq \frac{4(N-1)^{1/p} L C + 1}{S_\varepsilon}.
	\]
\end{itemize}
\end{theorem}

While the quantity $S_{\eps}$ in \cref{thm:lip}\,(ii) may not admit a closed-form expression, we can deduce more explicit bounds as follows. 

\begin{example}[Explicit bounds]\label{ex:explicitBounds}
  Choose a function $\psi\geq\varphi$ such that $\tilde{g}(x):=x\psi(x^{\beta})$ is strictly increasing with inverse denoted~$\tilde{g}_{\inv}$. Then $\tilde{g}_{\inv}\leq \tilde{f}_{\inv}$ and hence $1/S_{\eps} \leq 1/\tilde{g}_{\inv}(1/\eps)$, so that \cref{thm:lip}\,(ii) implies
	\[
	\OT_{f, \varepsilon} - \OT \leq (4(N-1)^{1/p} L C + 1) \frac{1}{\tilde{g}_{\inv}(1/\eps)}.
	\]
	We thus aim to choose $\psi$ so that $\tilde{g}_{\inv}$ has an explicit expression.
  As an example,  consider the $L^{\rho}$ regularization given by $f(x)=\frac{1}{\rho}(x^{\rho}-1)$ with $\rho>1$. Here $\varphi(x)=\frac{1}{\rho}x^{\rho-1}-\frac{1}{\rho x} \leq \frac{1}{\rho}x^{\rho-1} =:\psi(x)$. With this choice of $\psi$, we have $\tilde{g}(x)=\frac{1}{\rho}x^{(\rho-1)\beta+1}$ and the explicit inverse $\tilde{g}_{\inv}(x)=\rho x^{1/[(\rho-1)\beta+1]}$. As a result, for all $\eps\in(0,1]$,
	\[
	 \OT_{f, \varepsilon} - \OT \leq K\eps^{\frac{1}{(\rho-1)\beta+1}}, \quad K:=(4(N-1)^{1/p} L C + 1)/\rho.
	\]
\end{example}

\begin{remark}[On $\mu_{1}$]\label{rk:missingQuantRate}
  In \cref{thm:lip}, nothing is assumed about the quantization of~$\mu_{1}$. In an application, one would thus label~$\mu_{1}$ the marginal with the slowest quantization rate. In particular, for $N=2$ marginals on $\R^{d_{i}}$, we typically have $1/\alpha_{2}= d_{1} \wedge d_{2}$ by Remark~\ref{rk:quantizationRate}.
\end{remark}

\begin{proof}[Proof of \cref{thm:lip}]
	Let $\pi^*\in\Pi(\mu_1, \dots, \mu_N)$ be an optimizer of $\OT$. By our assumption, there exist empirical quantizations $\mu_i^{n_i}$ for the marginals $i=2, \dots, N$ such that $W_p(\mu_{i}^{n_{i}}, \mu_{i}) \leq C n_{i}^{-\alpha_{i}}$. We introduce a coupling $\pi\in\Pi(\mu_1, \dots, \mu_N)$ (depending on $n_{2},\dots,n_{N}$) as a double shadow: first, let $\tilde{\pi}$ be the shadow of $\pi^*$ onto $(\mu_1, \mu_2^{n_2}, \dots, \mu_N^{n_N})$; then, define $\pi$ as the shadow of $\tilde{\pi}$ onto $(\mu_1, \dots, \mu_N)$. Using the triangle inequality and \cref{le:shadow}, 
  \begin{align*}
  	W_p(\pi, \pi^{*}) 
  	&\leq W_p(\pi, \tilde{\pi}) + W_p(\tilde{\pi}, \pi^{*})
  	 \leq 2 \left(\sum_{i=2}^N W_p(\mu_{i}^{n_{i}}, \mu_{i})^p\right)^{1/p}.
	\end{align*}
	Combining this with our assumption~\eqref{eq:ccond}, we deduce
	\begin{align*}
	\int c \,d\pi - \int c \,d\pi^* %
	\leq 2L \left(\sum_{i=2}^N W_p(\mu_{i}^{n_{i}}, \mu_{i})^p\right)^{1/p} 
	\leq 2 LC \left(\sum_{i=2}^N n_{i}^{-\alpha_{i}p}\right)^{1/p}.
	\end{align*}
  On the other hand, \cref{le:shadow} again yields
  \begin{align*}
    D_f(\pi, P) \leq D_f(\tilde{\pi}, \mu_1 \otimes \mu_2^{n_2} \otimes \dots \mu_N^{n_N}).
	\end{align*}	
	As $\pi$ is an admissible coupling for $\OT_{f, \varepsilon}$, we have proved
	\begin{align}\label{eq:proofLip1}
	  \OT_{f, \varepsilon} - \OT \leq 2 LC \left(\sum_{i=2}^N n_{i}^{-\alpha_{i}p}\right)^{1/p} + \eps D_f(\tilde{\pi}, \mu_1 \otimes \mu_2^{n_2} \otimes \dots \mu_N^{n_N})
	\end{align} 
	and the last divergence term can be bounded by \cref{lem:boundfdiv}. In the remainder of the proof, we choose $n_{i}$ as a suitable function of $\eps$ to balance the decay of the two terms on the right-hand side of~\eqref{eq:proofLip1}.
	As $n_{i}$ is an integer, we need to deal with a rounding error: given $S\in[1,\infty)$, we define $\varrho(S)>0$ as
  \begin{align}\label{eq:roundingLipschitz1}
    \varrho(S) := \left(\frac{1}{N-1}\sum_{i=2}^N \frac{S^{p}}{\lfloor S^{1/\alpha_i} \rfloor^{\alpha_i p}} \right)^{1/p} %
  \end{align}
  so that $1\leq  \varrho(S) \leq 2^{\max_{i\geq2}\alpha_{i}} \leq 2$ and $\lim_{S\to\infty}\varrho(S)= 1$. We then have
  \begin{align}\label{eq:roundingLipschitz2}
    \left(\sum_{i=2}^N \lfloor S^{1/\alpha_i} \rfloor^{-\alpha_i p}\right)^{1/p} =  \frac{\varrho(S) (N-1)^{1/p}}{S} \leq \frac{2 (N-1)^{1/p}}{S}.
  \end{align}
	
  \noindent (i) Set $n_i = \lfloor \varepsilon^{-1/\alpha_i} \rfloor$ for $i=2, \dots, N$. For $S=S_{\eps}=1/\eps$, \eqref{eq:roundingLipschitz2} yields
  $$
     \left(\sum_{i=2}^N n_{i}^{-\alpha_{i}p}\right)^{1/p} =  \frac{\varrho(S_\varepsilon) (N-1)^{1/p}}{S_\varepsilon} \leq 2(N-1)^{1/p}\varepsilon,
  $$
  and \cref{lem:boundfdiv}\,(iii) bounds the divergence term by
  $$
    \eps D_f(\tilde{\pi}, \mu_1 \otimes \mu_2^{n_2} \otimes \dots \mu_N^{n_N}) \leq \eps \sum_{i=2}^N\log (n_i) 
    \leq \eps \sum_{i=2}^N \frac{1}{\alpha_i} \log\left(\frac{1}{\varepsilon}\right).
  $$
  In view of~\eqref{eq:proofLip1}, the claim follows. 
  
  \vspace{.3em}
  
  \noindent (ii) Set $n_i = \lfloor S_\varepsilon^{1/\alpha_i}\rfloor$ for $i=2, \dots, N$, where $S_\varepsilon$ was defined in the theorem. Similarly as in~(i),
	\[
	  \left(\sum_{i=2}^N n_{i}^{-\alpha_{i}p}\right)^{1/p} \leq \frac{\varrho(S_\varepsilon) (N-1)^{1/p}}{S_\varepsilon}  \leq 2 (N-1)^{1/p} \frac{1}{S_\varepsilon}.
	\]
  On the other hand, $S_\varepsilon \geq y_0^{1/\beta}+1$ implies $y_{0} \leq \prod_{i=2}^N n_i\leq S_\varepsilon^{\beta}$ by elementary arguments. Under ${\sf quant}^{em}_{p}(C,\alpha_i)$, 
	\cref{lem:boundfdiv}\,(ii) and monotonicity of $\varphi$ on $[y_0, \infty)$ yield
	\begin{align*}
	  \eps D_f(\tilde{\pi}, \mu_1 \otimes \mu_2^{n_2} \otimes \dots \mu_N^{n_N}) 
	  \leq \eps\varphi\left(\prod_{i=2}^N n_i\right) 
	  \leq \eps \varphi(S_\varepsilon^\beta) 
	  &=  \frac{\varepsilon \tilde{f}(S_\varepsilon)}{S_\varepsilon} \\
	  &= \frac{\varepsilon \tilde{f}(\tilde{f}_{\inv}(\frac{1}{\varepsilon}))}{S_\varepsilon} = \frac{1}{S_{\varepsilon}}
	\end{align*}
	and now the claim again follows from~\eqref{eq:proofLip1}. 
	For the claim under $N=2$, we use Lemma \ref{lem:boundfdiv}\,(i) instead of Lemma~\ref{lem:boundfdiv}\,(ii).
\end{proof}

\begin{remark}[On the constant]\label{rk:lipschitzBetterConst}
  The constant 4 in \cref{thm:lip}\,(i),(ii) can be replaced by $2\varrho(1/\eps)$ and $2\varrho(S_{\varepsilon})$, respectively, where $\varrho(\cdot)$ is defined in~\eqref{eq:roundingLipschitz1} and satisfies $1\leq \varrho(\cdot) \leq 2$. As $\varrho(S)=1+o(1/S)$, this improves the asymptotic constant for $\eps\to0$ in \cref{thm:lip} from 4 to 2. 
\end{remark} 

\begin{remark}[On the proof]\label{rk:lipschitzProof}
  In \cref{thm:lip} and its proof, the entropic case~(i) is treated separately from the general case~(ii) to obtain an expression that is more explicit and more in line with the literature. In fact, the bound in \cref{thm:lip}\,(ii) is slightly sharper even for the entropic divergence, as its proof is based on the optimal tradeoff between the transport and divergence terms: both have the same rate $1/S_\varepsilon$, whereas in the proof of~(i) they have differing rates $\eps$ and $\eps\log(1/\eps)$. However, $S_{\varepsilon} = \tilde{f}_{\inv}(\frac{1}{\varepsilon})$ does not admit an explicit expression in the entropic case, so we chose instead $S_{\varepsilon} =1/\eps$ to obtain an explicit statement. The leading-order term nevertheless turns out to be sharp; see \cref{pr:lipschitzSharp}.
\end{remark}

\subsection{Rate for Twice Differentiable Costs}

For the main result, we focus on the exponent $p=2$ for the Wasserstein metric and on  closed convex sets $X_i \subset \mathbb{R}^{d_i}$ endowed with the Euclidean norm~$|\cdot|$. We recall that $X=X_{1}\times\cdots\times X_{N}$ then also carries the Euclidean metric and write $c\in C^{2}(X)$ to indicate that $c$ is defined and twice continuously differentiable on a neighborhood of $X\subset \R^{d_{1}+\cdots+d_{N}}$.

For costs with bounded second derivative and an additional regularity condition, we shall improve upon the dimension-dependence in \cref{thm:lip} by a factor~$1/2$, at least for marginals of equal dimension. %
For that improvement, \eqref{eq:ccond} is too weak (as evidenced in \cref{pr:lipschitzSharp}). Instead, we shall use a martingale argument to achieve a full cancellation of the integrated first-order term in the Taylor expansion of~$c$. For this, we directly quantize an optimal transport, not just the marginals. In the following statement, its quantization rate~$\alpha$ is taken as given---we shall elaborate below on how to bound it in practice.

\begin{theorem}\label{thm:second}
	Let $X_i \subset \mathbb{R}^{d_i}$ be convex and $\mu_i \in \mathcal{P}_2(X_i)$ for $i=1, \dots, N$. Assume that $c\in C^{2}(X)$ has  bounded second derivative:
\begin{align}\label{eq:bddSecondDeriv}
  w^\top c''(x) w \leq B |w|^2 \quad \mbox{for all} \quad x, w \in X,\quad \mbox{for some} \quad B\geq0,
\end{align} 
 and that $\OT$ admits an optimal transport~$\pi^*$ satisfying ${\sf quant}_{2}(C,\alpha)$ for some $\alpha\in(0,1]$ and $C > 0$.
	\begin{itemize}
		\item[(i)] Let $f(x) = x \log(x)$. Then for all $\varepsilon \in (0, 1]$,%
		\[
		\OT_{f, \varepsilon} - \OT \leq  \frac{N-1}{2\alpha} \varepsilon \log\left(\frac{1}{\varepsilon}\right) + 8B C \varepsilon.
		\]		
		\item[(ii)] Let $N=2$, $f(x) = x \varphi(x)$ with $\varphi$ nondecreasing and concave, let $\beta = \frac{1}{2 \alpha}$ and $\tilde{f}(x) = x \varphi(x^\beta)$.
		Assume that for some $x_{0}\geq0$, $\tilde{f}$ is strictly increasing on $[x_0, \infty)$ with inverse  $\tilde{f}_{\inv}$. Set $S_{\varepsilon} = \tilde{f}_{\inv}(\frac{1}{\varepsilon})$, which satisfies $\lim_{\varepsilon\to0} S_{\varepsilon}=\infty$ and $\lim_{\varepsilon\to0} \varepsilon S_{\varepsilon}=0$.  Then for all $\varepsilon \in (0, \frac{1}{x_0}]$ small enough such that $S_\varepsilon \geq1$,%
		\[
		\OT_{f, \varepsilon} - \OT \leq \frac{8 BC + 1}{S_\varepsilon}.
		\]
	\end{itemize}
\end{theorem}

Before proving the theorem, we recall the martingale property of $W_{2}$-quantization; see, e.g., \cite[Proposition~5.1, p.\,139]{Pages.18} for a proof. This property and its interplay with the Taylor expansion in~\eqref{eq:taylor} below explain why our result is limited to $p=2$.

\begin{lemma}\label{le:martingale}
  Given a probability $\eta\in \cP_{2}(Y)$ on a Polish space~$Y$ and $n\geq1$, there exists $\eta^{n}\in \argmin_{\eta^n \in \mathcal{P}^n(Y)} W_2(\eta^n, \eta)$, called an optimal $W_2$-quantizer of~$\eta$ on~$n$ points. There is a coupling  $\theta \in \Pi(\eta^{n}, \eta)$ attaining $W_2(\eta^{n}, \eta)$, meaning that $\int |x-y|^2 \,\theta(dx, dy) = W_2(\eta^{n}, \eta)^{2}$, and it is a martingale: the kernel $\kappa$ in its disintegration $\theta = \eta^{n} \otimes \kappa$ satisfies $\int y \,\kappa(x, dy) = x$ for $\tilde{\pi}$-almost all $x$.
\end{lemma}

\begin{proof}[Proof of \cref{thm:second}.]
For $n\geq 1$, let  $\tilde{\pi}\in \cP(X)$ be an optimal $W_2$-quantizer of~$\pi^*$ on~$n$ points and let $\theta \in \Pi(\tilde{\pi}, \pi^*)$ be the coupling attaining $W_2(\tilde{\pi}, \pi^*)$; cf.\ \cref{le:martingale}. The martingale property of~$\theta$ 
  implies that 
  $
		\int h(x) \cdot (y-x) \,\theta(dx, dy) = 0
	$
	for any measurable function $h : X \rightarrow \R^{d_{1}+\cdots+d_{N}}$ of linear growth. As $c$ has bounded second derivative, its first derivative $c'$ has linear growth and thus
  \[
	  \int c'(x) \cdot (y-x) \,\theta(dx, dy) = 0.
	\]
		Considering the Taylor expansion of $c(y)$, this shows that the integral of the first-order term vanishes, and then the bound on the second derivative yields
	\begin{align}\label{eq:taylor}
	 \left|\int c \,d\pi^{*} - \int c \,d\tilde{\pi}\right| 
	  & = \left|\int (c(y)-c(x))\,\theta(dx, dy)\right|  \nonumber \\
		& \leq B \int |x-y|^2 \,\theta(dx, dy)
		= B W_2(\tilde{\pi}, \pi^*)^{2}.
  \end{align}  
  Denote by $\mu_i^n$ the marginal of $\tilde{\pi}$ on $X_i$ and by $\theta_i$ the marginal of $\theta$ on $X_i \times X_i$. We observe that $\theta_i \in \Pi(\mu_i^n, \mu_i)$ is again a martingale coupling.
  Furthermore, as we are using the Euclidean norm,
		\begin{align}\label{eq:thetai}
		\sum_{i=1}^N \int |x_i-y_i|^2 \,\theta_i(dx_i, dy_i) = \int |x-y|^2 \,\theta(dx, dy) = W_2(\tilde{\pi}, \pi^*)^{2}.%
		\end{align}
		Next, we construct a coupling $\pi\in \Pi(\mu_1, \dots, \mu_N)$ that is reminiscent of the shadow of $\tilde{\pi}$ but uses the kernels of $\theta_i$ instead of $W_{2}$-optimal transports between~$\mu_i^n$ and~$\mu_i$. Namely, decomposing $\theta_i = \mu_i^n \otimes K_i$ and writing  $K(x) := K_1(x_1) \otimes \dots \otimes K_N(x_N)$, we set $\gamma := \tilde{\pi} \otimes K\in\cP(X\times X)$ and define $\pi\in \Pi(\mu_1, \dots, \mu_N)$ as the second marginal of $\gamma$. Probabilistically speaking, this means that we take the (possibly dependent) components of the vector martingale $\theta$ and combine their laws into a new vector martingale~$\gamma$ with independent components. In particular,  $\gamma \in \Pi(\tilde{\pi}, \pi)$ is also a martingale coupling: $\int y_i \,K(x, dy) = \int y_i \,K_i(x_i, dy_i) = x_i$ for all~$i$ by the martingale property of $\theta_i$. Repeating the argument for~\eqref{eq:taylor} with~$\gamma$ instead of~$\theta$, inserting the definition of $\gamma$ and using~\eqref{eq:thetai}, we conclude that
	\begin{align*}
	\left| \int c \,d\pi - \int c \,d\tilde\pi\right| &\leq B \int |x-y|^2 \,\gamma(dx, dy)\\ &= B \sum_{i=1}^N \int |x_i - y_i|^2 \, \theta_i(dx_i, dy_i) = B W_2(\tilde{\pi}, \pi^*)^{2}.
	\end{align*}
	In view of~\eqref{eq:taylor}, the triangle inequality and the assumption on~$\pi^{*}$ then yield
	\begin{align}\label{eq:transportDiff}
	\int c \,d\pi - \int c \,d\pi^* \leq 2B W_2(\tilde{\pi}, \pi^*)^{2} \leq 2 B C n^{-2\alpha}.
  \end{align}
  On the other hand, by the data processing inequality (e.g., \cite[Lemma~1.6]{Nutz.20}), the construction of~$\pi$ implies
		\[
	D_f(\pi, P) \leq D_f(\tilde{\pi}, \mu_1^n \otimes \dots \otimes \mu_N^n).
	\]
	This bound is analogous to \cref{le:shadow} (indeed the reasoning is the same).
	
	The rest of the proof is analogous to \cref{thm:lip}. To deal with the rounding error, we now define $\varrho(S)$ for $S\in[1,\infty)$ as
  \begin{align}\label{eq:roundingSmooth1}
    \varrho(S) := \left(\frac{S^{\frac{1}{2\alpha}}}{\lfloor S^{\frac{1}{2\alpha}}\rfloor} \right)^{2\alpha}
  \end{align}
  so that $1\leq  \varrho(S) \leq 2^{2\alpha} \leq 4$ and $\lim_{S\to\infty}\varrho(S)= 1$. In particular,
  \begin{align}\label{eq:roundingSmooth2}    
    \lfloor S^{\frac{1}{2\alpha}}\rfloor^{-2\alpha} = \varrho(S) S^{-1} \leq  4S^{-1}.
  \end{align}

	\noindent (i) Let $n = \lfloor \varepsilon^{-\frac{1}{2\alpha}}\rfloor$. Then~\eqref{eq:transportDiff} and~\eqref{eq:roundingSmooth2}  for $S=S_{\eps}=1/\eps$ imply
	\begin{align*}
	\int c \,d\pi - \int c \,d\pi^* &\leq 2 BC \varrho(S_\varepsilon) S_\varepsilon^{-1} \leq 8BC \varrho(S_\varepsilon) \varepsilon
	\end{align*}
	while Lemma~\ref{lem:boundfdiv}\,(iii) yields
	$
	 D_f(\tilde{\pi}, \mu_1^n \otimes \dots \otimes \mu_N^n) \leq (N-1) \log(n),
	$
	completing the proof of~(i).
	
	\vspace{.3em}
	
	\noindent (ii) Here we define $n = \lfloor S_\varepsilon^{\frac{1}{2\alpha}}\rfloor$, then~\eqref{eq:transportDiff} and~\eqref{eq:roundingSmooth2} yield
	\[
	\int c \,d\pi - \int c \,d\pi^* \leq \frac{2 BC \varrho(S_\varepsilon)}{S_\varepsilon}\leq \frac{8 BC}{S_\varepsilon}
	\]
	while (recall $N=2$) Lemma~\ref{lem:boundfdiv}(i) yields
	$
	D_f(\tilde{\pi}, \mu_1^n \otimes \mu_2^n) \leq \varphi(n)
	$ 
	and thus
	\[
	\varepsilon D_f(\pi, P) \leq \varepsilon \varphi(n) \leq \frac{\varepsilon \varphi(S_\varepsilon^{\frac{1}{2\alpha}}) S_\varepsilon}{S_\varepsilon} = \frac{1}{S_\varepsilon},
	\]
		completing the proof.
\end{proof}

Similarly as in \cref{rk:lipschitzBetterConst}, the asymptotic constant in \cref{thm:second} can be improved from~8 to~2.

\begin{remark}[Relaxing $C^{2}$ condition]\label{rk:lessRegular}
\Cref{thm:second} immediately extends to slightly less regular costs: if $(c_n)_{n \in \mathbb{N}}$ is a sequence of cost functions satisfying the assumptions of \cref{thm:second} and $\lim_{n\to\infty}\|c_n - c\|_\infty = 0$ for some $c : \mathbb{R}^d \rightarrow \mathbb{R}$, then 
\[
\OT_{f, \varepsilon}(c) - \OT(c) \leq 2 \|c_n - c\|_\infty + \OT_{f, \varepsilon}(c^n) - \OT(c^n)
\]
as both $\OT_{f, \varepsilon}$ and $\OT$ are $1$-Lipschitz with respect to $\|\cdot\|_\infty$, so that \cref{thm:second} applies to $c$ as well.
\end{remark} 

We also have the following analogue of \cref{ex:explicitBounds}.

\begin{example}[$L^{\rho}$ regularization]\label{ex:explicitBoundSecond}
  For the $L^{\rho}$ regularization $f(x)=\frac{1}{\rho}(x^{\rho}-1)$ with $\rho>1$, \cref{thm:second}\,(ii) implies that for all $\eps\in(0,1]$,
	\[
	 \OT_{f, \varepsilon} - \OT \leq K\eps^{\frac{1}{(\rho-1)\beta+1}}, \quad K:= (8 BC + 1)/\rho,
	\]
	by the same algebra as in \cref{ex:explicitBounds}. (Of course, $\beta$ now has a different definition).
\end{example}

\begin{remark}[Comparison with \cref{thm:lip}] \label{rk:comparison}
  Let $N=2$ for simplicity. As any quantization of the coupling~$\pi^{*}$ induces quantizations for its marginals, it is clear that $\alpha \leq \alpha_{2}$. In the best case, we have $\alpha = \alpha_{2}$, and then \cref{thm:second} yields an improvement of $1/2$ over \cref{thm:lip}. Note that $\alpha = \alpha_{2}$ will typically be the case if $d_{1}=d_{2}=:d$ and the support of~$\pi^{*}$ is also $d$-dimensional---more on this in a moment.
  
  On the flip side, as \cref{thm:second} implicitly quantizes all the marginals, there is no immediate benefit to having a faster rate for one marginal as in \cref{rk:missingQuantRate}. Thus there are situations where \cref{thm:lip} actually yields a better rate, especially if $d_{1}> 2d_{2}$. But of course, $d_{1}=d_{2}$ is the most important setting.
\end{remark}

To obtain a good result from \cref{thm:second}, we need to know that $\OT$ admits an optimal transport~$\pi^*$ satisfying ${\sf quant}_{2}(C,\alpha)$ for some good~$\alpha$. Indeed, ${\sf quant}_{2}(C,\alpha)$ holds trivially for $1/\alpha=d_{1}+\cdots+d_{N}$ (under a moment condition), but that does not yield the desired improvement over \cref{thm:lip}. On the other hand, suppose that $\pi^{*}$ is given by a Lipschitz transport map over $X_{1}$, then $\pi^{*}$ inherits the quantization rate from~$\mu_{1}$, so that $1/\alpha = d_{1}$. The existence of such a map has been studied intensely in the regularity theory of optimal transport, see \cite{Caffarelli.92,Caffarelli.96} and the literature thereafter. However, the conditions are known to be very restrictive \cite{Loeper.09,MaTrudingerWang.05},
 and clearly a Lipschitz map can almost never be expected for unbounded marginals. On the other hand, as emphasized in~\cite{McCannPassWarren.12}, a lower dimensional structure does not require a transport map at all.
 
 In the following, we provide some results for $N=2$ marginals, and remark briefly on the multi-marginal case. Generally speaking, any result on the structure of optimal transports can be combined with \cref{thm:second}. The next result covers the most important example---the quadratic cost defining 2-Wasserstein distance---under a minimal condition on the marginals (which includes many situations where no coupling is given by a map).

\begin{lemma}\label{le:quadCostQuantRate}
  Consider $c(x,y)=|x-y|^{2}$ on $\R^{d}\times\R^{d}$ with marginals $\mu_{1},\mu_{2}\in\cP_{2+\delta}(\R^{d})$ for some $\delta>0$. Then any optimal transport satisfies ${\sf quant}_{2}(C,1/d)$ for some $C>0$.
\end{lemma} 

\begin{proof}
  Let $\Delta=\{(x,x): x\in \R^{d}\}$ be the diagonal and $\proj^{\Delta}: \R^{2d}\to \Delta$ the Euclidean orthogonal projection. Let $\pi\in\Pi(\mu_{1},\mu_{2})$ be an optimal transport, then $\pi\in\cP_{2+\delta}(\R^{2d})$ due to the assumption on the marginals. Define the pushforward measure 
  \[
    \eta := \proj^{\Delta}_{\#} \pi
  \]
  which is concentrated on~$\Delta$; we claim that $\eta$ satisfies ${\sf quant}_{2}(C,1/d)$. Consider the rotated coordinates $(u,v)$ given by
  \[
    u=\frac{x+y}{\sqrt{2}}, \quad v=\frac{x-y}{\sqrt{2}}
  \]
  in which $\Delta=\{(u,0): u\in\R^{d}\}$ and $\proj^{\Delta}$ can be written as $(u,v)\mapsto(u,0)$. Thus $\eta$ can be seen as a measure on $\R^{d}$ and with that identification,
  \[
    \int |u|^{2+\delta}\,d\eta = \int |(u,v)|^{2+\delta}\,d\pi = \int |(x,y)|^{2+\delta}\,d\pi<\infty.
  \]
  By \cref{rk:quantizationRate}, $\eta\in\cP_{2+\delta}(\R^{n})$ implies that~$\eta$ satisfies ${\sf quant}_{2}(C,1/d)$. %
  
  To show the same rate for $\pi$, we use Minty's trick~\cite{Minty.62} along the lines of~\cite{AlbertiAmbrosio.99}. Recall that the support $\Gamma :=\spt \pi$ is $c$-cyclically monotone (e.g.,~\cite{Villani.09}), which for quadratic cost means
  \[
    \br{x'-x, y'-y} \geq0 , \quad (x,y),(x',y')\in\Gamma.
  \]
  In the rotated coordinates, this implies that 
  \[
    |v'-v| \leq |u'-u|, \quad (u,v),(u',v')\in\Gamma.
  \]  
  In particular, $u=u'$ implies $v=v'$, meaning that $\proj^{\Delta}$ admits an inverse map $\ell: \proj^{\Delta}(\Gamma)\to\Gamma$, $(u, 0) \mapsto (u, v)$ and moreover $\ell$ is $\sqrt{2}$-Lipschitz. By Kirszbraun's theorem, we can extend $\ell$ to a $\sqrt{2}$-Lipschitz map $\Delta\to\R^{d}\times\R^{d}$, still denoted~$\ell$. Note that $\pi=\ell_{\#} \eta$ and any quantization of $\eta$ on~$\Delta$ pushes forward to a quantization of~$\pi$. In view of the $\sqrt{2}$-Lipschitz property, we conclude that $\pi$ satisfies ${\sf quant}_{2}(\sqrt{2} C,1/d)$.
\end{proof} 

The following combines \cref{le:quadCostQuantRate} with \cref{thm:second} and \cref{ex:explicitBoundSecond}.

\begin{corollary}[Quadratic cost]\label{co:quadratic}
  Consider $c(x,y)=|x-y|^{2}$ on $\R^{d}\times\R^{d}$ with marginals $\mu_{1},\mu_{2}\in\cP_{2+\delta}(\R^{d})$ for some $\delta>0$. 
	\begin{itemize}
		\item[(i)] Let $f(x) = x \log(x)$. There exists $K>0$ such that 
		\[
		\OT_{f, \varepsilon} - \OT \leq \frac{d}{2} \varepsilon \log\left(\frac{1}{\varepsilon}\right) + K \varepsilon, \qquad \varepsilon \in (0, 1].
		\]		
		\item[(ii)] Let $f(x)=\frac{1}{\rho}(x^{\rho}-1)$ with $\rho>1$. There exists $K>0$ such that 
	  \[
	  \OT_{f, \varepsilon} - \OT \leq K\eps^{\frac{1}{(\rho-1)d/2+1}}, \qquad \varepsilon \in (0, 1].
	  \]
	\end{itemize}
\end{corollary}

Next, we aim to generalize \cref{le:quadCostQuantRate} from quadratic to more general costs. Following~\cite{McCannPassWarren.12}, the basic idea is that a fairly generic cost is locally equivalent to a perturbation of the quadratic cost after a change of coordinates. Let $X_1,X_2 \subset \mathbb{R}^{d}$ be convex and $c\in \cC^{2}(X)$. We say that $c$ is \emph{nondegenerate} if $D_{xy}^{2}c(x,y)$ is invertible for all $(x,y)\in X$. Here $D_{xy}^{2}c(x,y)$ denotes the $d\times d$ matrix $[\partial^{2}_{x_{i}y_{j}}c(x,y)]_{1\leq i,j\leq d}$. We follow the terminology of~\cite{McCannPassWarren.12}; the condition is called~(A2) in~\cite{MaTrudingerWang.05} while~\cite{CarlierPegonTamanini.22} calls such~$c$ infinitesimally twisted.

If the support can be covered by finitely many such local coordinate changes, we obtain the same quantization rate as in the quadratic case. In particular, this holds for compact support.

\begin{lemma}\label{le:nondegenerateCompact}
   Let $X_1,X_2 \subset \mathbb{R}^{d}$ be convex and let $c\in \cC^{2}(X)$ be nondegenerate. If $\mu_{1},\mu_{2}$ are compactly supported, then any optimal transport satisfies ${\sf quant}_{2}(C,1/d)$ for some $C>0$.
\end{lemma} 

For a proof, see Steps~1 and~2 in the proof of \cref{le:uniformlyNondegcosts} below. Next, we address the unbounded case; here we assume that nondegeneracy holds in a uniform sense (which is automatic in the compact case) and achieve a rate arbitrarily close to $1/d$, under sufficient integrability. The proof is a combination of the proofs of \cref{le:quadCostQuantRate} and \cite[Theorem~1.1]{McCannPassWarren.12} with a cut-off argument.
We denote by~$\|M\|$ the operator norm of the matrix~$M$.

\begin{lemma}\label{le:uniformlyNondegcosts}
   Let $X_1,X_2 \subset \mathbb{R}^{d}$ be convex and let $c\in \cC^{2}(X)$ be nondegenerate. Suppose that $D_{xy}^{2}c(x,y)$ is uniformly continuous and $\|D_{xy}^{2}c\|$, $\|(D_{xy}^{2}c)^{-1}\|$ are bounded on~$X$. Let $d'>d$. If $\mu_{1},\mu_{2}\in\cP_{q}(\R^{d})$ for $q:=2\frac{d'+d}{d'-d}$, then any optimal transport satisfies ${\sf quant}_{2}(C,1/d')$ for some $C>0$.
\end{lemma} 

\begin{proof}
  Let $\pi$ be an optimal transport. Whenever a subprobability $\nu$ is given, we denote by $\widetilde\nu=\nu/\nu(X)$ its normalized measure.

  \vspace{.3em}
  
  \noindent \emph{Step~1}. Consider a cube $Q=([-r,r]^{2d}+\{(x_{0},y_{0})\})\cap X$ centered at $(x_{0},y_{0}) \in \spt\pi$. We show that for $r$ sufficiently small, $\widetilde{\pi|_{Q}}$ satisfies ${\sf quant}_{2}(C,1/d)$ with a constant $C$ independent of $(x_{0},y_{0})$.
  Let $M:=D_{xy}^{2}c(x_{0},y_{0})\in\R^{d\times d}$ and $G(x, y) := -c(x, -M^{-1}y) - x\cdot y$. Then
  \begin{align*}
  D_{xy}^2 G(x, y) &= D^2_{xy} c(x, -M^{-1} y) M^{-1} - \1_{n} \\ &= D^2_{xy} c(x, -M^{-1}y) M^{-1} - D^2_{xy} c(x_0, y_0) M^{-1}
  \end{align*}
  and hence
  \[
  \|D_{xy}^{2}G(x,y)\| \leq \|M^{-1}\| \|D^2_{xy} c(x, -M^{-1}y) - D^2_{xy} c(x_0, y_0)\|.
  \]
  As $D^2_{xy} c$ is uniformly continuous and $\|(D_{xy}^2 c)^{-1}\|$ is uniformly bounded, we can thus choose $r \in (0, 1)$ independent of $(x_0, y_0)$ such that $\|D_{xy}^{2}G(x,y)\| \leq \frac{1}{2}$ for all $(x,y)\in\R^{d}\times\R^{d}$ with $(x, -M^{-1}y) \in Q$.
  
  Consider $(x,y),(x',y')$ such that $(x, -M^{-1}y), (x', -M^{-1}y') \in Q \cap \spt \pi$. Then the $c$-cyclical monotonicity of $\spt \pi$ yields
  \[
    c(x, -M^{-1}y) + (x', -M^{-1}y') \leq c(x, -M^{-1}y) + c(x', -M^{-1}y')
  \]
  or equivalently 
  \begin{equation}\label{eq:ccyclical}
  x\cdot y + G(x, y) + x' y'+G(x', y') \geq x \cdot y' + G(x, y') + x' \cdot y + G(x', y).
  \end{equation}
  Next, we use a second change of coordinates
  \[
  u=\frac{x+y}{\sqrt{2}}, \quad v=\frac{x-y}{\sqrt{2}}.
  \]
  Closely following the proof of \cite[Theorem~1.2]{McCannPassWarren.12}, using \eqref{eq:ccyclical} with $\Delta x := x' - x$, $\Delta y := y' - y$, $\Delta u := u' - u$, $\Delta v := v' - v$  leads to
  \[
  \Delta x \cdot \Delta y + \Delta x \cdot \int_0^1 \int_0^1 D^{2}_{xy}G(x + s \Delta x, y + t \Delta y)\Delta y \, ds \, dt \geq 0
  \]
  and hence $\Delta x \cdot \Delta y \geq -\frac{1}{2} |\Delta x| |\Delta y|$ as $\|D_{xy}^{2}G\| \leq \frac{1}{2}$ along the integration domain. Noting that $\Delta y \sqrt{2} = \Delta u + \Delta v$ and $\Delta x \sqrt{2} = \Delta u - \Delta v$, we deduce
  \begin{align*}
  |\Delta u|^2 - |\Delta v|^2 = 2 \Delta x \cdot \Delta y &\geq - |\Delta x| |\Delta y| \\
  &\geq - \frac{1}{2} (|\Delta x|^2 + |\Delta y|^2) = - \frac{1}{2} (|\Delta u|^2 + |\Delta v|^2)
  \end{align*}
  and thus
  \begin{equation}\label{eq:lipinverse}
  |\Delta v| \leq \sqrt{3} |\Delta u|.
  \end{equation}
  Consider the composition $a=a_{3}\circ a_2 \circ a_1$ of the linear maps 
  \[a_1 : (x, -M^{-1} y) \mapsto (x, y), \quad a_{2}: (x,y)\mapsto(u,v), \quad a_3: (u, v) \mapsto u.
  \]
  Clearly, the image $I=a(\R^{2d})$ is a $d$-dimensional linear subspace.
  Defining $\eta := a_{\#} \widetilde{\pi|_{Q}}$, we see that $\spt \eta$ is a bounded subset of $I$. Its diameter admits a bound depending only on~$r$ and the Lipschitz constant of $a$, and the latter is independent of $(x_{0},y_{0})$ as~$\|M\|=\|D_{xy}c(x_{0},y_{0})\|$ is uniformly bounded. Recall from \cref{rk:quantizationRate} that a measure on $\R^{d}$ with bounded support satisfies ${\sf quant}_2(C_{0}, 1/d)$ with a constant~$C_{0}$ depending only on~$d$ and the diameter of the support (note that the diameter bounds any moment). As a result, $\eta$ satisfies ${\sf quant}_2(C_{0}, 1/d)$ with a constant~$C_{0}$ independent of $(x_0, y_0)$.
  
  The map $a$ admits a Lipschitz inverse $\ell: a(Q \cap \spt \pi)\to Q \cap \spt \pi$, with a Lipschitz constant~$L$ independent of $(x_0, y_0)$ due to the boundedness of $\|(D_{xy}^2 c)^{-1}\|$ and~\eqref{eq:lipinverse}. Again, by Kirszbraun's theorem, $\ell$ extends to a Lipschitz map $\ell: I\to \R^{d}\times\R^{d}$ with the same Lipschitz constant. As $\widetilde{\pi|_{Q}} = \ell_{\#} \eta$, we deduce that $\widetilde{\pi|_{Q}}$ satisfies ${\sf quant}_2(C, 1/d)$ for $C=LC_{0}$.

  \vspace{.3em}
  
  \noindent \emph{Step~2}.  We start with a general observation about sums. Let $\nu_{1},\dots,\nu_{m}$ be subprobabilities with a cumulative mass of at most one and suppose that each $\widetilde{\nu_{i}}$ satisfies ${\sf quant}_{2}(C,\alpha)$ for $n\geq1$. Consider the quantization problem for the sum $\nu=\sum_{i=1}^{m}\nu_{i}$, which can be seen as the convex combination $\sum_{i=1}^{m} \nu_{i}(X) \widetilde{\nu_{i}}$ of probability measures (and the zero measure, if necessary). Noting that given $n=km$ points, we can allocate~$k$ points to each of the $\widetilde{\nu_{i}}$, it is easy to see that $\widetilde\nu$ satisfies ${\sf quant}_{2}(m^{\alpha}C,\alpha)$ for all $n\in \{m,2m,\cdots\}$, and thus for all $n\geq m$ after increasing the constant $C$.
	
  For $N\in\N$, consider $R=Nr$ and the cube $Q_{R}=[-R,R]^{2d}$, which can be divided into $m:=N^{2d}$ small cubes of the type in Step~1. Combining Step~1 with the observation about sums, we see that $\widetilde{\pi|_{Q_{R}}}$ satisfies ${\sf quant}_{2}((R/r)^{2}C ,1/d)$ for $n\geq (R/r)^{2d}$. 
  
  We note that if the marginals are compactly supported, $Q_{R}$ contains $\spt\pi$ for $R$ sufficiently large, so that $\pi$ satisfies ${\sf quant}_{2}(C ,1/d)$ after increasing~$C$. For the noncompact case, we use the following cut-off.

  \vspace{.3em}
  
  \noindent \emph{Step~3}.  For $n\geq1$, choose $R=R(n)$ as
	\[
	  R(n) := r \left\lfloor n^{\frac{1}{d}-\frac{1}{d'}}\right\rfloor^{1/2}.
	\]
	Note $\lim_{n\to\infty}R(n)=\infty$ and $n\geq (R(n)/r)^{2d}$ and $(R(n)/r)^{2}n^{-1/d}\leq n^{-1/d'}$. Writing $\pi_{n} :=\pi|_{Q_{R(n)}}$, the above shows that there exist $\nu_{n}\in\cP^{n}(\R^{2d})$ such that $W_{2}(\nu_{n},\widetilde{\pi_{n}})\leq C n^{-1/d'}$.
	On the other hand, consider $\pi-\pi_{n}$, which is supported outside $[-R(n),R(n)]^{2d}$. As a consequence,
  \[
   \int |z|^{2} \,d(\pi-\pi_{n})\leq R(n)^{-\gamma}\int |z|^{2+\gamma} \,d\pi
  \]	
  for any $\gamma\geq 0$. Choose $\gamma:=\frac{4d}{d'-d}=\frac{4}{d'}(\frac{1}{d}-\frac{1}{d'})^{-1}$; then $R(n)^{-\gamma}\leq C' n^{-2/d'}$ for a constant $C'>0$ and as $2+\gamma=2\frac{d'+d}{d'-d}$, the integral is finite by our assumption on the marginals. Quantizing $\pi-\pi_{n}$ by a single point mass at the origin,  we then see with the result for $\pi_{n}$ that $\pi$ satisfies ${\sf quant}_{2}(C ,1/d')$ for a (different) constant $C$.
\end{proof} 

\Cref{le:nondegenerateCompact} and \cref{le:uniformlyNondegcosts} have immediate corollaries similar to \cref{co:quadratic}; we omit the statements for brevity.

The nondegeneracy condition can be extended to the multi-marginal transport problem, and  is used in \cite[Theorem~2.2]{Pass.15} to bound the dimension of the support of an optimal transport. However, as noted by the author, the condition is no longer generic when $N>2$, and indeed, some quite reasonable multi-marginal problems only have solutions of larger dimension \cite[Remark~2.13]{Pass.15}. On the other hand, we do expect that our results extend to $N>2$ for particular costs like those in \cite{GangboSwiech.1998}.
In any event, \cref{thm:second} separates such regularity issues from the convergence analysis, so that any available regularity result from optimal transport theory can be applied directly.

\begin{remark}%
  As mentioned in the Introduction, \cite{CarlierPegonTamanini.22} previously obtained the constant $d/2$, for compactly supported marginals with uniformly bounded Lebesgue densities, and also showed its sharpness (cf.\ \cref{se:sharpness}). Unlike in our result, the upper bound in~\cite{CarlierPegonTamanini.22} does not require nondegeneracy. Interestingly, Minty's trick is also used in \cite{CarlierPegonTamanini.22}, but it is employed in the proof of the sharpness rather than in the upper bound as in the present work. Above, we worked on the primal problem and used Minty's trick to estimate the dimension of optimal couplings. Whereas in \cite{CarlierPegonTamanini.22}, the authors work on the Kantorovich potentials of the dual problem to derive the upper bound, giving a quadratic control on the integrated difference between a $\lambda$-convex function and its first-order Taylor expansion.
\end{remark}

\section{Sharpness}\label{se:sharpness}

In this section we show that the upper bounds obtained in the preceding section are sharp in certain cases. Throughout, we focus on $N=2$ marginals and divergences given by $f(x)=x\log(x)$ and $f(x) = \frac{1}{\rho} (x^\rho - 1)$.
Lower bounds for $\OT_{f, \varepsilon} -\OT$ are naturally obtained from the dual problem of $\OT_{f, \varepsilon}$. %

\begin{lemma}\label{le:dualLower}
  Let $\hat h_{i}\in L^{1}(\mu_{i})$, $i=1,2$ be Kantorovich potentials for~$\OT$ and $\hat c (x,y):=c(x,y)-\hat h_{1}(x)-\hat h_{2}(y)$ for $(x,y)\in X_{1}\times X_{2}$. Let $f^*(y) := \sup_{x \geq 0} [xy - f(x)]$ for $y\in\R$ and $f^*_\varepsilon(y) := \varepsilon f^*(\frac{1}{\varepsilon} y)$. Then
  \begin{align*}	
	  \OT_{f, \varepsilon} -\OT
	  &\geq \sup_{a \in \mathbb{R}} \left( a - \int f^*_{\eps}\left(a-\hat c\right) \,d(\mu_1 \otimes \mu_2) \right)\\
	&\geq \sup_{a \in \mathbb{R}} \left(a - f^*_{\eps}(a)\int  \eins_{a \geq \hat c} \,d(\mu_1 \otimes \mu_2)  \right) - \eps f^*(0).
  \end{align*}
\end{lemma}

\begin{proof}
  Recall (e.g., \cite{EcksteinPammer.22, TerjekGonzalez.22}) the duality 
		\begin{align*}
	\OT_{f, \varepsilon} &= \sup_{h_1, h_2} \int h_1(x)+h_2(y)-f_{\varepsilon}^*\big(h_1(x)+h_2(y)-c(x, y)\big) \,\mu_1(dx)\mu_2(dy)
	\end{align*}
	where the supremum ranges over $h_{i}\in L^{1}(\mu_{i})$. As $\OT=\sum_{i=1}^{2} \int \hat h_{i}\,d\mu_{i}$, choosing  $h_{1}=\hat h_{1} + a$ and $h_{2}=\hat h_{2}$ yields
	\begin{align*}
	\OT_{f, \varepsilon} -\OT
	&\geq \sup_{a \in \mathbb{R}} \left( a - \int f_\varepsilon^*(a-\hat c) \,d(\mu_1 \otimes \mu_2) \right).
  \end{align*}
  As $f_{\eps}^*$ is nondecreasing,
	$\hat c \geq 0$ and $f_{\eps}^*(0)=\eps f^*(0) \geq -\eps f(1)=0$, we also have
	$
	\int f_\varepsilon^*(a-\hat c) \,d(\mu_1 \otimes \mu_2) \leq f_{\varepsilon}^*(a)\int  \eins_{a \geq \hat c} \,d(\mu_1 \otimes \mu_2) + f_\varepsilon^*(0)
  $, leading to the second inequality.
\end{proof}

Turning to the sharpness of the Lipschitz result (\cref{thm:lip}), it was observed in \cite[Example~3.3]{CarlierPegonTamanini.22} that the leading-order term $\eps\log(1/\eps)$ is sharp in the entropic case for the distance cost on~$\R$. Part~(i) below is a simple extension of that result to $d$ dimensions equipped with the $L^{1}$-metric as cost, showing that the dependence on the dimension (or equivalently the quantization rate) is also sharp. For $L^\rho$ regularization, we show in~(ii) that the leading term has the sharp order and in particular the correct dimension dependence. %
Regarding the relation between dimension and quantization rate, recall from~\cref{rk:quantizationRate} that $\alpha_{2}=1/d$ for absolutely continuous marginal~$\mu_{2}\in\cP(\R^{d})$.

\begin{proposition}[Sharpness of \cref{thm:lip}]\label{pr:lipschitzSharp}
	Let $X_1 = X_2 = \mathbb{R}^d$ with $\mu_1 = \mu_2$ the uniform distribution on $[0, 1]^d$ and $c(x, y) = \sum_{i=1}^{d} |x_{i}-y_{i}|$.
	\begin{enumerate}
	\item Let $f(x)=x \log(x)$. Then for all $\eps>0$,
	  \begin{align*}	
	  \OT_{f, \varepsilon} - \OT 
	  &\geq d \eps \log(1/\eps) - (2^{d}-1)\eps.
  \end{align*}
  In particular, the leading term matches the bound in \cref{thm:lip}\,(i).
  \item Let $f(x) = \frac{1}{\rho} (x^\rho - 1)$ for some $\rho > 1$. %
  Then
  \[
	  \OT_{f, \varepsilon} - \OT \geq K \varepsilon^{\frac{1}{(\rho-1)d+1}} + O(\eps)
	\]
	for a constant $K>0$. In particular, the leading term has the same exponent as the bound deduced from  \cref{thm:lip}\,(ii) in \cref{ex:explicitBounds}.
\end{enumerate}  
\end{proposition}

\begin{proof}
  (i) Here $f^*(x)=e^{x}-1$. Recalling the normalizing constant $\int_{\R} e^{\frac{|u-v|}{\eps}}\,du =2\eps$ of the Laplace distribution, 
	\begin{align*}	
	  \int e^{\frac{a-c}{\eps}} \,d(\mu_1 \otimes \mu_2) = e^{\frac{a}{\eps}} \prod_{i=1}^{d} \int_{[0,1]^{2}} e^{\frac{|x_{i}-y_{i}|}{\eps}}\,dx_{i}dy_{i} \leq e^{a/\eps}(2\eps)^{d},
	\end{align*}
  and thus \cref{le:dualLower} (with $\hat h_{1} = \hat h_{2}=0$) shows
  \begin{align*}	
	  \OT_{f, \varepsilon} - \OT
	  &\geq \sup_{a} \left(a - 2^{d}\eps^{d+1} e^{a/\eps} + \eps \right).
  \end{align*}
  Choosing $a=d \eps \log(1/\eps)$, the right-hand side equals $d \eps \log(1/\eps) - (2^{d}-1)\eps$.

  \vspace{.3em}
  \noindent (ii) Here $f^*(y) = \frac{1}{q} y_+^q + \frac{1}{\rho}$ for $q := \frac{\rho}{\rho-1}$, so that
	\[
	   f^*_\varepsilon(a)=\eps f^*(a/\eps) = \frac{a^q}{q\eps^{q-1}} + \frac{\eps}{\rho}, \quad a\geq0.
	\]
	The definition of~$c$ shows that $\eins_{a \geq c}\leq\prod_{i=1}^{d} \eins_{\{|x_{i}-y_{i}|\leq a\}}$ and thus 
	\[
	  \int \eins_{a \geq c} \,d(\mu_1 \otimes \mu_2) \leq \prod_{i=1}^{d} \int_0^1\int_0^1 \eins_{a \geq |x_i-y_i|} \,dx_i dy_i = (2a-a^{2})^{d} \leq (2a)^{d}
	\]
	for $a\in[0,1]$, with the last bound valid for $a\geq0$. \Cref{le:dualLower} thus yields 
	\begin{align}\label{eq:lowerPowerBound}
	\OT_{f, \varepsilon} - \OT
	&\geq \sup_{a \in\R_{+}} \left(a- 2^df_\varepsilon^*(a) a^d - \varepsilon f^*(0)\right) \\
	&=\sup_{a \in\R_{+}} \left(a- 2^d \frac{a^{d+q}}{q\eps^{q-1}} -\frac{2^d \eps a^d}{\rho} - \varepsilon\right). \nonumber
	\end{align}
	Setting $a := k \varepsilon^{\frac{1}{(\rho-1)d+1}}$, where $k > 0$ is such that  $K:=(k-2^d k^{q+d}/q)>0$, we deduce
	$
	\OT_{f, \varepsilon} -\OT
	\geq  K \varepsilon^{\frac{1}{(p-1)d+1}} +O(\eps)
	$ as claimed.
\end{proof}

\begin{remark}\label{rk:quadraticSharp}
  We can similarly show the sharpness of \cref{co:quadratic}\,(ii) for quadratic cost. Namely, let $c(x, y) = |x-y|^{2}=\sum_{i=1}^{d} |x_{i}-y_{i}|^{2}$. Going through the proof of 
  \cref{pr:lipschitzSharp}, we now have $\eins_{a \geq c}\leq\prod_{i=1}^{d} \eins_{\{|x_{i}-y_{i}|\leq \sqrt a\}}$, and thus
  $
	\OT_{f, \varepsilon} -\OT
	\geq  K \varepsilon^{\frac{1}{(p-1)d/2+1}} +O(\eps).
	$
	A more general (if much more involved) argument for a general class of marginals is given below.
\end{remark} 	

Indeed, we can establish the sharpness of \cref{thm:second} for a general class of marginals and costs. For the entropic case, it is well known that the leading term $\frac{d}{2} \varepsilon \log\left(\frac{1}{\varepsilon}\right)$ is sharp for quadratic cost $c(x,y)=|x-y|^{2}$ on $\R^{d}\times\R^{d}$ when the marginals are sufficiently regular \cite{Chizat2020Faster,ConfortiTamanini.19,Pal.19}. Very recently, \cite{CarlierPegonTamanini.22} showed that this term is sharp for the broad class of nondegenerate (as defined before \cref{le:nondegenerateCompact}) costs and regular marginals; their result is stated in~(i) below for completeness. The core of the proof in~\cite{CarlierPegonTamanini.22} is a quadratic detachment estimate for the Kantorovich potentials. In~(ii), we apply their technique to divergences $f(x) = \frac{1}{\rho} (x^\rho - 1)$ to show sharpness of the leading order in \cref{thm:second}\,(ii).

\begin{proposition}[Sharpness of \cref{thm:second}]\label{pr:secondSharp}
  For $i=1, 2$, let $X_i \subset \mathbb{R}^{d}$ be convex and compact and let $\mu_i \in \mathcal{P}(X_i)$ have bounded Lebesgue density. Let $c\in \cC^{2}(X)$ be nondegenerate.
	\begin{enumerate}
	\item Let $f(x)=x \log(x)$. Then
	  \begin{align*}	
	  \OT_{f, \varepsilon} - \OT 
	  &\geq \frac{d}{2} \eps \log(1/\eps) + O(\eps).
  \end{align*}
  In particular, the leading term matches the bound in \cref{thm:second}\,(i).
  \item Let $f(x) = \frac{1}{\rho} (x^\rho - 1)$ for some $\rho > 1$. Then
  \[
	  \OT_{f, \varepsilon} - \OT \geq K \varepsilon^{\frac{1}{(\rho-1)d/2+1}} + O(\eps)
	\]
	for a constant $K>0$. In particular, the leading term has the same exponent as the bound deduced from \cref{thm:second}\,(ii) in \cref{ex:explicitBoundSecond}.
\end{enumerate}  
\end{proposition}

\begin{proof}
  See \cite[Proposition~4.4]{CarlierPegonTamanini.22} for~(i). To show (ii), we argue that there exist constants $C_{0}, C > 0$ such that
\begin{equation}\label{eq:lbmain}
\OT_{f, \varepsilon} \geq \OT + \sup_{a \leq C_{0}}\left(a - C f_{\varepsilon}^*(a) a^{d/2} - \max\{0, f_{\varepsilon}^*(0)\}\right).
\end{equation}
  This bound is similar to~\eqref{eq:lowerPowerBound} but with different constants, and implies the claim along the same lines. To show~\eqref{eq:lbmain}, we will apply \cref{le:dualLower} with optimal potentials $(\hat h_1, \hat h_2)$. The latter can be chosen to be continuous, so that $\hat c$ is also continuous. The main difficulty is to bound $\int  \eins_{a \geq \hat c} \,d(\mu_1 \otimes \mu_2)$.
  Following the proof of \cite[Proposition~4.4]{CarlierPegonTamanini.22}, we find a finite open cover $A = \cup_{i=1}^n A_i$ of the compact set $\{\hat{c} = 0\} \cap (X_1 \times X_2)$
  satisfying the following:
	\begin{itemize}
		\item[(a)] On the compact $B:= (X_1 \times X_2) \backslash A$ we have $\hat{c} > C_{0}$ for some $C_{0} > 0$.
		\item[(b)] There exist $r,C_1 > 0$ such that for all $i \in \{1, \dots, n\}$, for some $r_v\in\R^{d}$ depending only on $v \in \mathbb{R}^d$,
		\[
		\int_{A_i} \eins_{a \geq \hat{c}} \,d(\mu_1 \otimes \mu_2) \leq C_1 \int_{B_r} \int_{B_r} \eins_{a \geq |u - r_v|^2/4} \,du dv,
		\]
		where $B_r \subset \mathbb{R}^d$ is the ball of radius $r > 0$ around the origin.
	\end{itemize}
	Bounding the inner integral in~(b) according to
		\begin{align*}
		\int_{B_r} \eins_{a \geq |u - r_v|^2/4} \,du &\leq \int_{\mathbb{R}^d} \eins_{a \geq |u|^2/4} \,du 
		\leq \prod_{i=1}^d \int_{\mathbb{R}} \eins_{|u_i| \leq 2 \sqrt{a}} \,du \leq 4^d a^{d/2},
		\end{align*}
		we obtain 
		\[
		\int_{A\cap (X_1 \times X_2)} \eins_{a \geq \hat{c}} \,d(\mu_1 \otimes \mu_2) \leq C a^{d/2}
		\]
		for a constant~$C>0$. In view of~(a), this shows
		\begin{equation}\label{eq:bounda}
		  \int_{X_1 \times X_2} \eins_{a \geq \hat{c}} \,d(\mu_1 \otimes \mu_2) \leq C a^{d/2} \quad \mbox{for}\quad a\leq C_{0}
		\end{equation}
		and now~\eqref{eq:lbmain} follows by \cref{le:dualLower}.
\end{proof}

\appendix
\section{Appendix}\label{se:appendix}

The following is well known in the entropic case \cite[Section~5]{Nutz.20}. For completeness, we provide an extension to the $f$-divergences under consideration.

\begin{proposition}\label{pr:finiteDivergenceOT}
  We have $\OT_{f, \varepsilon} - \OT = O(\eps)$ if and only if there exists an optimal transport plan $\pi^{*}$ for~$\OT$ with $D_{f}(\pi^{*},P)<\infty$.
\end{proposition} 

\begin{proof}
  If there exists an optimal transport plan~$\pi^{*}$ with finite divergence, clearly $\OT_{f, \varepsilon} - \OT \leq \eps D_{f}(\pi^{*},P)=O(\eps)$. Conversely, let $\pi_{\eps}$ be an optimizer of $\OT_{f, \varepsilon}$. 
  If $\OT_{f, \varepsilon} - \OT = O(\eps)$, it follows that $\sup_{\eps\in (0,1]}D_{f}(\pi_{\eps},P)<\infty$. As~$f$ has superlinear growth, the densities $d\pi_{\eps}/dP$ are then uniformly integrable; in particular, there exists a weak$^{*}$-convergent sequence $d\pi_{\eps_{n}}/dP$, meaning that $(\pi_{\eps_{n}})$ converge set-wise. The limit $\pi_{0}$ is again a coupling. We have $\int c\, d\pi_{0}\leq \liminf \OT_{f, \varepsilon_{n}}=\OT$ by a generalized Fatou's lemma~\cite[p.\,231]{Royden.68} and the growth condition on~$c$, showing that  $\pi_{0}$ is an optimal transport. 
  The same Fatou's lemma  shows $D_{f}(\pi_{0},P)\leq \liminf D_{f}(\pi_{\eps_{n}},P)<\infty$, completing the proof.
\end{proof} 

The following extension of~\cref{thm:lip} was prompted by a question of by G.~Carlier; see also the similar~\cite[Remark~3.2]{CarlierPegonTamanini.22}.

\begin{remark}[Extension of~\cref{thm:lip} beyond Lipschitz]\label{rk:modOfCont}
  Fix $p=1$ and replace~\eqref{eq:ccond} by 
\begin{align}\label{eq:modOfCont}
  \left|\int c \,d(\pi - \tilde{\pi})\right| \leq \omega(W_1(\pi, \tilde{\pi})),
\end{align}
  where $\omega : \mathbb{R}_+ \rightarrow \mathbb{R}_+$ is an increasing and concave modulus of continuity.
To motivate this, note that if the function~$c$ itself has modulus of continuity~$\omega$, then choosing $\theta\in\Pi(\pi,\tilde{\pi})$ attaining $W_1(\pi, \tilde{\pi})$ yields
\begin{align*}
\left|\int c \,d(\pi - \tilde{\pi})\right| 
&\leq \int |c(x)-c(y)| \,\theta(dx, dy) \\
&\leq \int \omega(d_{X,1}(x,y)) \,\theta(dx, dy) \leq \omega(W_1(\pi, \tilde{\pi}))
\end{align*} 
by Jensen's inequality.
Going through the proof of~\cref{thm:lip} with~\eqref{eq:modOfCont}, we obtain instead of~\eqref{eq:proofLip1} that
	\begin{align*}\label{eq:proofModOfCont}
	  \OT_{f, \varepsilon} - \OT \leq 2 \omega\left(C \sum_{i=2}^N n_{i}^{-\alpha_{i}}\right) + \eps D_f(\tilde{\pi}, \mu_1 \otimes \mu_2^{n_2} \otimes \dots \mu_N^{n_N})
	\end{align*} 
  and can then optimize the choice of~$n_{i}$. For instance, in the entropic case, we would take $S_\varepsilon = \omega^{-1}(1/\varepsilon)$; then the first term is again of order $\varepsilon$ while the divergence term is of order $\varepsilon\log(\omega^{-1}(1/\varepsilon))$. For $N=2$ and $c(x,y)=d_{X,1}(x,y)^{r}$ with $0<r<1$, we end up with 
  \[
	\OT_{f, \varepsilon} - \OT \leq \frac{1}{r\alpha_2} \varepsilon \log\left(\frac{1}{\varepsilon}\right) + K\varepsilon.
	\]
	It is worth noting the formal similarity with \cref{thm:second}\,(i) which would correspond to~$r=2$.
\end{remark}

\bibliography{stochfin}
\bibliographystyle{abbrv}

\end{document}